\documentclass[reqno,a4paper]{amsart}
\usepackage[utf8]{inputenc}
\usepackage[T1]{fontenc}
\usepackage[british]{babel}
\usepackage{amsmath,amsthm,amssymb}
\usepackage[foot]{amsaddr}

\usepackage{mathrsfs}

\usepackage{graphicx}
\usepackage{wrapfig}
\usepackage{subfig}
\usepackage{enumerate}
\usepackage[usenames,dvipsnames]{color}
\usepackage{dsfont}
\usepackage{microtype}

\usepackage{nomencl}

\usepackage{etoolbox}

\usepackage{mathtools}
\mathtoolsset{showonlyrefs,showmanualtags}

\usepackage{pstricks,pst-plot,pst-math} 
\usepackage{pstricks-add} 

\calclayout

\usepackage{esint}

\newtheorem{thm}{Theorem}[section]

\newtheorem{lem}[thm]{Lemma}
\newtheorem{prop}[thm]{Proposition}

\theoremstyle{definition}

\theoremstyle{remark}
\newtheorem{rem}[thm]{Remark}

\numberwithin{equation}{section}

\newcommand{\DeclareAutoPairedDelimiter}[3]{%
	\expandafter\DeclarePairedDelimiter\csname Auto\string#1\endcsname{#2}{#3}%
	\begingroup\edef\x{\endgroup
		\noexpand\DeclareRobustCommand{\noexpand#1}{%
			\expandafter\noexpand\csname Auto\string#1\endcsname*}}%
	\x}

\DeclareAutoPairedDelimiter{\abs}{\lvert}{\rvert}
\DeclareAutoPairedDelimiter{\norm}{\lVert}{\rVert}
\DeclareAutoPairedDelimiter{\bra}{(}{ )}
\DeclareAutoPairedDelimiter{\pra}{[}{]}
\DeclareAutoPairedDelimiter{\set}{\{}{\}}
\DeclareAutoPairedDelimiter{\skp}{\langle}{\rangle}

\DeclareMathAlphabet{\mathup}{OT1}{\familydefault}{m}{n}
\newcommand{\dx}[1]{\mathop{}\!\mathup{d} #1}

\DeclareMathOperator*{\supp}{supp}

\newcommand{\N}{\mathds{N}}
\newcommand{\R}{\mathds{R}}

\newcommand{\cL}{\ensuremath{\mathcal L}}

\newcommand{\sL}{\ensuremath{\mathscr L}}

\newcommand{\tzeta}{\ensuremath{\tilde \zeta}}

\usepackage[colorinlistoftodos,prependcaption,textsize=tiny]{todonotes}

\usepackage[pdfusetitle]{hyperref}
\definecolor{darkblue}{rgb}{0,0,0.6}
\hypersetup{
    pdfauthor={Simon Eberle},
    colorlinks=true,       
    linkcolor=darkblue,          
    citecolor=darkblue,        
    filecolor=darkblue,      
    urlcolor=darkblue           
}

\author{Simon Eberle$^1$}
\address{$^1$Faculty for mathematics, University of Duisburg-Essen.}

\email{simon.eberle@uni-due.de}


\title[Front blocking versus propagation with drift]{Front blocking versus propagation in the presence of drift term varying in the direction of propagation}

\let\rho\varrho
\let\epsilon\varepsilon

\begin{document}

\begin{abstract}
\noindent
In this paper we derive quantitative conditions under which a compactly supported drift term depending on the direction of propagation blocks a traveling wave solution or lets it pass almost unchanged.
We give explicit conditions on the drift term for blocking as well as almost unchanged propagation in one spatial dimension.

\end{abstract}
	\maketitle

\section{Introduction}

The object of this paper is the investigation of transition fronts in one spatial dimension subject to a compactly supported drift term in the direction of propagation and depending on the direction of propagation.
In such a setting classical traveling waves are impossible and the following two cases are possible

\begin{enumerate}	
	\item \label{ref_blocking} blocking, i.e. no propagation of anything front like.
	\item \label{ref_unchanged_propagation}`almost unchanged propagation', i.e. the effect of the drift term fades out for large time (up to maybe a possible shift of the front). 
\end{enumerate}	

In this paper we are able to give sufficient (a priori) conditions on the drift term, such that case \ref{ref_blocking} or \ref{ref_unchanged_propagation}   occur. 

To the best knowledge of the author there are no results on that matter available yet. So we hope to offer a first partial understanding on what happens to traveling waves subject to drift disturbance that varies in the direction of propagation.

The investigation of traveling waves in cylinders, also subject to drift, has been done in depth in the seminal paper \cite{traveling_fronts_in_cylinders}. However the drift term has been required to be independent of the direction of propagation, in order to allow for classical traveling waves. Since then the notion of traveling waves has been broadened to more general media, i.e. pulsating fronts for periodic media \cite{BerestyckiHamelPeriodicExcitableMedia} and the very general transition fronts for very general media \cite{BerestyckiHamelGeneralizedTransitionFronts}.
In recent years there have been investigations of existence and non existence of transition fronts in outer domains with a compactly supported obstacle \cite{MatanoObst}, in cylinders with varying nonlinearity \cite{Zlatos,change_of_speed_1} and, with respect to this work especially interesting, in opening or closing cylinders \cite{front_blocking,suden_opening,change_of_speed2}.

The subject of this paper are entire solutions of the generalized initial value problem

	\begin{align} \label{DiffEqu}
	\begin{cases}
	\partial_t u(t,x) - \partial_{xx} u(t,x) +k(x) \partial_{x} u(t,x) = f\bra{u(t,x)} &\text{ for all } (t,x) \in \R \times \R, \\
	u(t,x) - \phi(x+ct) \rightarrow 0 \text{ as } t \rightarrow -\infty &\text{ uniformly in } \R,
	\end{cases}
	\end{align}
	where  $k \in C_c^\infty(\R)$, $\supp k \subset [-x_0,0]$, $x_0>0$ and $f$ is a bistable nonlinearity i.e. 
\begin{align}
\label{F_1} &f \in C^{2}([0,1]), \\
\label{F_2} &f(0)  = 0, f(1) = 0, \\
\label{F_3}&f'(0) <0 , f'(1) <0, \\
\label{F_4}  &f<0 \text{ on } (0, \theta), f >0 \text{ on } (\theta, 1) \text{ for some } \theta \in (0,1), 
\\
\label{F_8}& \int \limits_0^1 f(u) \dx{u}  >0 .
\end{align}

In \eqref{DiffEqu} we mean with $\phi$ and $c$ the unique (up to translation) traveling wave profile $\phi$ and unique speed $c>0$ for the one-dimensional problem (only dependent on the nonlinearity $f$) that solves
\begin{align} \label{ODE_travelling_wave}
\begin{cases}
\phi^{\prime \prime}(z) - c \phi^\prime(z) + f(\phi(z)) =0, \\
\phi(- \infty) = 0, \phi(+ \infty)=1, \\
0<\phi(z) < 1\text{ for all } z \in \R ,\\
\phi'(z) >0 \text{ for all } z \in \R.
\end{cases}
\end{align}
For details and proofs see e.g. \cite{FifeMcLeod}.

	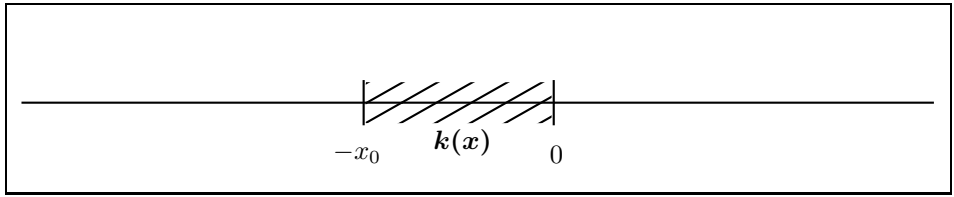
\begin{figure}[!h]
		\psset{xunit=1cm,yunit=1cm}
		\frame{
			\begin{pspicture*}
			(-5.2,-1.2)(7.1,1.3)
		\psframe[fillstyle=vlines,hatchsep=0.2,hatchangle=120,linecolor=white](-0.5,-0.3)(2,0.3)
			\psline(-5,0)(7,0)
				\psline(-0.5,-0.3)(-0.5,0.3)
					\psline(2,-0.3)(2,0.3)
			\psline(0,0)(0,-0)
			\psline(2,0)(2,-0)
			\rput[lb](0.3,-0.7){ \boldmath$ k(x)$}
			\rput[lb](-0.9,-0.8){$-x_0$}
			\rput[lb](1.95,-0.8){$0$}
			\end{pspicture*}
		}
		\caption{Infinite cylinder with transition zone}
	\end{figure}
In order to make sure that we are not investigating the empty set of solutions of \eqref{DiffEqu}, let us state the following proposition.

\begin{prop}[Existence and uniqueness]\label{prop:existence_uniqueness}
Let $f$ be as above. Then there is a \emph{unique} entire solution $u(t,x)$ of \eqref{DiffEqu} such that $0 < u(t,x) <0$ and $\partial_t u(t,x) >0$ for all $(t,x) \in \R \times \R$.
\end{prop}
The proof for Proposition \ref{prop:existence_uniqueness} can be obtained almost literally copying the proof of Theorem 2.1 in \cite{MatanoObst} or Appendix A in \cite{change_of_speed_1}.(For further details see Appendix \ref{section:existence_and_uniqueness}.)
\\[0.5cm]
For blocking (case \ref{ref_blocking}) we are able to give an explicit criterion involving the net drift and some term that takes into account the concentrations of $k$ as formulated in 
\begin{thm} \label{theorem:blocking}
	There is a constant $C(f)>0$ only depending on $f$ such that if 
	\begin{align} \label{blocking_condition}
	C(f) > \exp \bra { - \int_\R k(s) \dx{s}  }   \bra {  2+ \max \set {   14, \bra {    \int \limits_{-x_0}^0    \exp \bra {  \int \limits_{-x_0}^t k(s) \dx{s}  }  \dx{t}            }^{\frac{1}{2}}                }       }^2
	\end{align}
	holds, the unique solution of \eqref{DiffEqu} is blocked to the left, i.e. there exists a stationary supersolution $w: \R \rightarrow \R$ of \eqref{DiffEqu} such that
	\begin{align}
	u(t,x) \leq w(x) \quad \text{ for all } t \in \R, x \in \R
	\end{align}
	and $w(x) \rightarrow 0$ as $x \rightarrow -\infty$.
\end{thm}

On the other hand, if the positive part of the drift term $k$ is small enough as well as its support, we prove `almost unchanged propagation' (case \ref{ref_unchanged_propagation}). Note that `almost unchanged propagation' is much stronger than being a transition front.

\begin{thm} \label{main_theorem}
	There is a constant $C(f,x_0)>0$  (only depending on $f$ and $x_0$) such that if
	$k^+:= \max \set { \max_{s \in \R} k(s),0}$ is small enough to satisfy
	\begin{align} \label{condition_theorem_propagation}
	k^+  \leq C(f,x_0),
	\end{align}
	then the unique solution $u$ of \eqref{DiffEqu} converges to a traveling wave with profile $\phi$ and speed $c$, i.e.
	\begin{align}
	u(t,x) - \phi(x+ct +\beta) \rightarrow 0 \text{ as } t \rightarrow +\infty \text{ uniformly in } \R, 
	\end{align}
	where $\beta \in \R$ is a constant shift.
\end{thm}

The strategy of the proof of Theorem  \ref{theorem:blocking} is to construct the stationary supersolution $w$ as local minimizer of an appropriate functional in some weighted Sobolev space. The main observation is that \eqref{DiffEqu} becomes variational if the drift term is encoded in some weight. With this `trick' one can use ideas from \cite{front_blocking}, where the authors show that a neck can be introduced into a given tube in such a way that propagation gets blocked by constructing a stationary supersolution that vanishes behind the neck.

The main problem in the proof of Theorem \ref{main_theorem} is to achieve propagation, or more precisely, that something similar to the front passes by the disturbance. The rest is establishing a priori estimates and stability results for a Lyapunov function argument (similar to \cite{change_of_speed_1,FifeMcLeod}).

The paper is organized as follows. First we clarify assumptions and notation. Then, for the sake of completeness we shortly address the question of existence and uniqueness in section \ref{section:existence_and_uniqueness}. In section \ref{section:blocking} we give the strategy of the proof and the proof of Theorem \ref{theorem:blocking}. In section \ref{section:existence/propagation} we prove Theorem \ref{main_theorem}. 
\begin{rem}
	\begin{enumerate}
	\item
	The proof of Theorem \ref{theorem:blocking} can be adapted to the more general case 
	\begin{align} 
	\begin{cases}
	\partial_t u - \Delta u +k \cdot \nabla u = f(u) &\text{ in } D, \\
	\frac{\partial u}{\partial \nu} = 0 &\text{ on } \partial D, \\
	u(t,x) - \phi(x_1+ct) \rightarrow 0 \text{ as } t \rightarrow -\infty &\text{ uniformly in } D,
	\end{cases}
	\end{align}
	where $D:= \R \times \Omega$ is a cylindrical domain, $k \in C_c^\infty(\bar D, \R^n)$, $\supp k \subset [-x_0,0] \times \bar{\Omega}$ and $x_0>0$, if $k$ has a `logarithmic potential', i.e. there is $\psi\in C^\infty(D)$ such that
	\begin{align}
	k= -\nabla \log(\psi) = - \frac{\nabla \psi}{\psi} .
	\end{align}
	\item 
	Note that $14$ in \eqref{blocking_condition} is only a technical constant which is not optimal. It can be improved, even using the same method. But since there is no reason to assume that the criterion in Theorem \ref{theorem:blocking} is optimal, we did not bother too much to optimize constants.
	\item Concerning Theorem \ref{theorem:blocking} we understand $\int_{\R}k$ as the net drift that equally weights the positive and the negative part of $k$ against each other  and we understand  $ \int \limits_{-x_0}^0    \exp \bra {  \int \limits_{-x_0}^t k(s) \dx{s}  }  \dx{t}     $ as a measure for the concentration of the drift term. To illustrate that let us look at the following family of drift functions:
	\begin{align}
		k_\epsilon := \frac{K}{\epsilon} \chi_{[-\epsilon,0]} \quad , \text{ where } \epsilon >0 \text{ and } K>0. 
	\end{align}
	A calculation reveals that for all $\epsilon >0$
	\begin{align}
		\exp \bra { - \int \limits_\R k_\epsilon    } = \exp \bra {-K } .
	\end{align} 
	Because $\supp k_\epsilon = [-\epsilon,0]$, we have that $x_0(\epsilon)= \epsilon$. Therefore we get by another direct calculation
	\begin{align}\label{concentration_example}
		\int \limits^ 0_{-x_0(\epsilon)} \exp \bra { \int \limits_{-x_0(\epsilon)}^t k_\epsilon(s) \dx{s}    }\dx{t} = \frac{\epsilon}{K} \bra { \exp(K) -1  }.
	\end{align}
	Let now $K>0$ be large enough such that 
	\begin{align} 
		\frac{1}{K} \bra { \exp(K) -1    } > 14^2
	\end{align}
	then
	\begin{align}
		\exp \bra { - \int_\R k_1(s) \dx{s}  }   \bra {  2+ \max \set {   14, \bra {    \int \limits_{-x_0(1)}^0    \exp \bra {  \int \limits_{-x_0(1)}^t k_1(s) \dx{s}  }  \dx{t}            }^{\frac{1}{2}}                }       }^2 \\
		= \bra { 2 \exp \bra {  - \frac{K}{2}} + \frac{1}{\sqrt{K}} \bra { 1 - \exp \bra { - K }}^\frac{1}{2}        }^2  \sim \frac{1}{K} \quad \text{ for } K \text{ large.}
	\end{align}
	This implies that the criterion in \eqref{blocking_condition} can be fulfilled and we have not given an  'empty' condition.
	
  We have already seen in \eqref{concentration_example} that  $ \int \limits_{-x_0}^0    \exp \bra {  \int \limits_{-x_0}^t k(s) \dx{s}  }  \dx{t}     $  measures concentration in the sense that the expression becomes small if $\abs{\supp k}$ becomes small while $\int_\R k$ is fixed.
	
	If $K$ is very large and $\epsilon >0$ is very small such that $k$ is very concentrated (think for example of 		$k_\epsilon \to \delta_0 K$ as  $\epsilon \to 0$ in the sense of measures)
	then for $\epsilon$ small enough we have that
	\begin{align}
	&\exp \bra { - \int_\R k_\epsilon(s) \dx{s}  }   \bra {  2+ \max \set {   14, \bra {    \int \limits_{-x_0}^0    \exp \bra {  \int \limits_{-x_0}^t k_\epsilon(s) \dx{s}  }  \dx{t}            }^{\frac{1}{2}}                }       }^2 \\&= 16^2 \exp(-K).
	\end{align}
	This suggests that concentration helps for the blocking criterion \eqref{blocking_condition} to be met.

	To the author this seems plausible, taking into consideration that in \cite{suden_opening} the authors showed that sudden opening of a channel leads to front blocking.
	\item From \eqref{blocking_condition} we see directly that, for the necessary blocking criterion to hold, it is necessary that
	\begin{align}
		\frac{C(f)}{16^2} > \exp \bra { - \int_\R k  } .
	\end{align}
	This tells us that - at least for the criterion to be met - concentration of (positive) $k$ alone cannot do it, if this condition is violated.
	\item  The present result is related to the question of blocking and propagation in widening channels as discussed in \cite{front_blocking} in such a way that \eqref{DiffEqu} can be understood as singular limit problem of the problem investigated in \cite{front_blocking} for diameter going to zero.
		\end{enumerate}
\end{rem}

\section*{Acknowledgement}
We thank Prof. Dr. G.S. Weiss for fruitful discussions.

\section{A necessary condition for propagation / a sufficient condition for blocking} \label{section:blocking}

The objective of this section shall be the proof of Theorem \ref{theorem:blocking}.
The proof of our result on blocking will rely mainly on the observation that problem \eqref{DiffEqu} is variational with functional
\begin{align}
J_A(w) = \int \limits_A \bra { \frac{1}{2} \abs { w' }^2 +F(w)     } \psi(x) \dx{x},
\end{align}
where $F(t):= \int_t^1 f(s) \dx{s}$ and $\psi$ is a solution of the differential equation
\begin{align} \label{Def:psi}
\psi'(x) = -k(x) \psi(x) .
\end{align}
Here we have one degree of freedom, let's say $\psi(-x_0)>0$.
Therefore $\psi(x)>0$ for all $x \in \R$ and $\psi$ can be used as weight function. With this trick of encoding the drift term in a weight function we are now in the position to use variational techniques to construct a local minimizer of the functional $J$ that will then be extended to a stationary supersolution.
The strategy of this proof is inspired by the strategy used in \cite{front_blocking} where the authors show that a thin neck can be introduced into a given channel in such a way that a traveling wave gets blocked.

Let us briefly describe our strategy in the following. The goal is to construct $w$ such that
\begin{align}
	u(t,x) \leq w(x) \quad \text{ for all } x \in \R, t \in \R
\end{align}
and $w(x)\rightarrow 0$ as $x \rightarrow -\infty$, which will be possible if condition \eqref{blocking_condition} is met.

To make $J$ well defined and to ensure that $F$ grows quadratically at infinity $f$ shall be extended to a function $f \in C^2(\R)$ such that
\begin{itemize}
	\item $f(s) >0$  if $s<0$ , $f(s)<0$ if $s>1$,
	\item $f$ is asymptotically linear, i.e. $f'(s) \rightarrow d_+ <0$ as $s \rightarrow \infty$ and $f'(s) \rightarrow d_-<0$ as $s \rightarrow -\infty$,
	\item $\norm {f''}_{L^\infty([0,1])}   = \norm {f''}_{L^\infty(\R)}$.
\end{itemize}
In order to construct such a supersolution we

\begin{enumerate}
	\item first show that for any $R<-x_0-1$ (arbitrary but fixed) there is $\delta(f, \psi,k, a)>0$ independent of $R$ such that 
	\begin{align}
		J_{(R,a)} (w) > J_{(R,a)} (w_0)
	\end{align} 
	for all $w \in H^1_{0,1} ((R,a), \psi \dx{x})$ such that $\norm {w-w_0}_{H^1((R,a), \psi \dx{x})} = \delta$ and $w_0(x) := \frac{x}{a} \chi_{[0,a]}(x)$, $a>0$ is an auxiliary constant that can be chosen in an optimal way (depending on $f$) and we understand
	\begin{align}
		H^1_{0,1}((c,d), \psi \dx{x}) := \set { v \in H^1((c,d), \psi \dx{x}) : v(c)=0, v(d)=1 }
	\end{align}
	(where boundary values are understood in the sense of traces).
	From this we can conclude by the direct method, that there is a local minimizer $w_R \in H^1_{0,1} ((R,a), \psi \dx{x})\cap \set { \norm {w-w_0}_{H^1((R,a), \psi \dx{x})} \leq \delta }$ that is a weak solution of
	\begin{align}
		-w_R'' -\frac{\psi'}{\psi} w_R' = f(w_R) \text{ in } (R,a) \\
		w_R(R)=0, w_R(a)=1
	\end{align}	
	with $0<w_R<1$ in $(R,a)$ (by comparison principle).
\item In a next step we pass to the limit $R \rightarrow -\infty$ exploiting that $\delta$ is independent of $R$ and show that the limit $w_\infty$ solves
\begin{align}
-w_\infty'' -\frac{\psi'}{\psi} w_\infty' = f(w_\infty) \text{ in } (R,a) \\
w_\infty(-\infty)=0, w_\infty(a)=1.
\end{align}
and it follows for such a solution (by the strong maximum principle) that $0<w_\infty<1$ in $(-\infty,a)$.
\item In the last step we show that if we extend $w_\infty$ by $1$ into $[a,\infty)$ it is a supersolution of \eqref{DiffEqu}.
\end{enumerate}

\begin{figure}[!h]
	\psset{xunit=1cm,yunit=1cm}
	\frame{
		\begin{pspicture*}
		(-5.2,-1.2)(7.1,1.3)
		\psframe[fillstyle=vlines,hatchsep=0.2,hatchangle=120,linecolor=white](-0.1,-0.3)(2,0.3)
		\psline(-5,0)(7,0)
		\psline(-0.1,-0.3)(-0.1,0.3)
		\psline(2,-0.3)(2,0.3)
		\rput[lb](0.5,-0.7){ \boldmath$ k(x)$}
		\rput[lb](-0.4,-0.8){$-x_0$}
		\rput[lb](1.9,-0.8){$0$}
		\rput[lb](3.9,-0.8){$a$}
		\rput[lb](-4,-0.8){$R$}
		\psline[linestyle=dashed](3.9,-0.3)(3.9,0.3)
		\psline[linestyle=dashed](-4,-0.3)(-4,0.3)
		\rput[lb](-2,-0.8){$-x_0-1$}
		\psline[linestyle=dashed](-1,-0.3)(-1,0.3)
		\end{pspicture*}
	}
	\caption{Infinite cylinder with transition zone}
\end{figure}
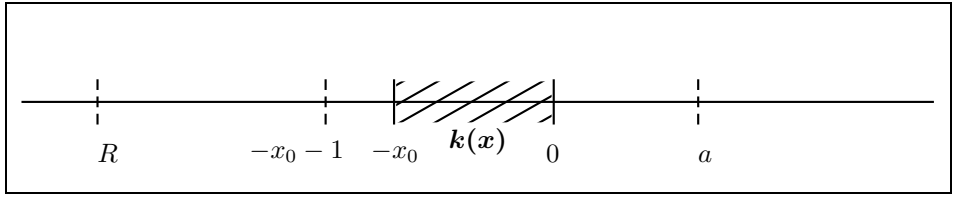

\begin{prop} \label{proposition:blocking:fixed_R}
	There is a constant $C(f)>0$ (only depending on $f$) such that if it satisfies condition \eqref{blocking_condition}, then for all $R<-x_0-1$ there is 
	\begin{align} \label{Def:delta}
	\delta := \frac{\alpha \sqrt{\psi(-x_0)}}{ \norm {f''}_{L^\infty}\bra {2+ \max \set { 14, \bra { \int \limits_{-x_0}^0 \exp \bra { \int \limits_{-x_0}^t k(s) \dx{s}     } \dx{t} } }^\frac{1}{2}   }} >0
	\end{align}
	(independent of $R$)
	such that for any $R<-x_0-1$ there is a local minimizer 
	\begin{align}
w_R \in H^1_{0,1}((R,a), \psi \dx{x}) \cap \set { \norm {w_R - w_0}_{H^1{((R,a), \psi \dx{x})}}    < \delta  }
\end{align}
	 of 
	\begin{align}
	J_{(R,a)} 
	\end{align}
	in $H^1_{0,1}((R,a), \psi \dx{x}) \cap \set { \norm {w_R - w_0}_{H^1{((R,a), \psi \dx{x})}}    \leq \delta  }$.
	The constant
	\begin{align} \label{Def:alpha}
		\alpha:= \min \set { \frac{1}{4} , -\frac{f'(0)}{4}   } >0
	\end{align}
	only depends on $f$, $w_0(x) := \frac{x}{a} \chi_{[0,a]}(x)$ and $a>0$ is an auxiliary constant.
\end{prop}

In order to prove this we will split up $(R,a)$ into the part $(0,a)$ where $\psi$ is constant and $w_0$ is linear and $(R,0)$ where $w_0 \equiv 0$ and $\psi$ does encode the behaviour of $k$.

On the second subset we will exploit the following Lemma.

\begin{lem} \label{lemma:auxiliary_lemma_fixed_R}
	With $\delta$ and $\alpha$ given as in \eqref{Def:delta}, \eqref{Def:alpha} in Proposition \ref{proposition:blocking:fixed_R}, it holds that for all $w \in H^1_0((R,0),\psi \dx{x}) \cap \set { \norm {w-w_0}_{H^1((R,0), \psi \dx{x}) } = \norm {w}_{H^1((R,0), \psi \dx{x})} \leq \delta  }$
	\begin{align}
	J_{(R,0)}(w) \geq J_{(R,0)} (w_0) + \alpha \norm { w  }^2_{H^1((R,0), \psi \dx{x}))},
	\end{align}
	where
	$H^1_0((R,0), \psi \dx{x})) : = \set { w \in  H^1((R,0), \psi \dx{x})) : w(R)=0 } $.
\end{lem}

\begin{proof}[Proof of the Lemma]
	First by a Taylor expansion of $F$ we find that 
	
	\begin{align} \label{Taylor_expansion_of_functional}
	J_{(R,0) }(w) = \int \limits_R^ 0 \bra {  \frac{1}{2}  \abs {w'}^2 +F(0) +F'(0) w + \frac{1}{2} F''(0) w^2 + \eta(w) w^3     }  \psi \dx{x} ,
	\end{align}
	where $ \abs {\eta(w)}  \leq \norm {f''}_{L^\infty}   $.
	We can rewrite this as 
	\begin{align}
	J_{(R,0)}(w) = J_{(R,0)} (w_0) + \int \limits_R^ 0 \bra { \frac{1}{2} \abs {w'}^2 - \frac{1}{2} f'(0) w^2 + \eta(w) w^3    } \psi .
	\end{align}
	It is immediate that
	\begin{align}
	\int \limits_R^ 0 \bra { \frac{1}{2} \abs {w'}^2 -      \frac{1}{2} f'(0) w^2    } \psi \geq \min \set { \frac{1}{2} , -\frac{f'(0)}{2}   }  \norm {w}^2_{H^1((R,0), \psi \dx{x})} .
	\end{align}
	It remains to absorb the last term in the Taylor expansion into this. 
	In order to do so we will use that 
	\begin{align}
	\norm { w  }^3_{L^3((R,0), \psi \dx{x})} \leq \sigma \bra {\psi,k,x_0 , \norm { w  }_{H^1((R,0), \psi \dx{x})}  } \norm { w  }^2_{H^1((R,0), \psi \dx{x})} ,
	\end{align}
	where $\sigma$ is independent of $R$ and $\sigma \rightarrow 0$ as $ \norm { w  }_{H^1((R,0), \psi \dx{x})} \rightarrow 0$.

	First of all on $(R,-x_0)$ in order to avoid any domain-dependency of the embedding constant (since we want to decrease $R$ later), we use only the fundamental theorem of calculus:
	\begin{align} \label{L4_norm_on_0_R}
	\begin{split}
	\norm {w }^4_{L^4((R,-x_0), \psi \dx{x})} &=	\int \limits_R^{-x_0} w^4 \psi \leq \int \limits_R^{-x_0} w^2 \psi \norm {w^2}_{L^\infty(R,-x_0)} \\
	&\leq \norm {w }^2_{L^2((R,-x_0), \psi \dx{x})} \int \limits_R^{-x_0} \abs {  (w^2)' }  \\
	&\leq \norm {w }^2_{L^2((R,-x_0), \psi \dx{x})}  \frac{1}{\min_{[R,-x_0]} \psi} \int \limits_R^{-x_0} \abs {  (w^2)' } \psi  \\
	&\leq  \frac{1}{\psi(-x_0)} \norm {w }^4_{H^1(R,-x_0), \psi \dx{x})}.
	\end{split}
	\end{align}

	For the Sobolev embedding on $(-x_0,0)$ we need a trace estimation and therefore we need a cut off function
	\begin{align}
	\zeta \in C_c^1(\R) \text{ s.t. } \zeta =1 \text{ in } \bra {-x_0 -\tfrac{1}{4},-x_0 +\tfrac{1}{4}   } \text{ and } \zeta = 0 \text{ in } \R \setminus \bra {-x_0-\tfrac{1}{2} , -x_0 +\tfrac{1}{2}} .
	\end{align}
	Note that $\zeta$ can be chosen such that $\norm {\zeta}_{C^1(\R)} \in \pra {1,7}$. This will be assumed in the following.
	
	Then we have
	\begin{align} \label{trace_estimate_w_at_0}
	\begin{split}
	w(-x_0) &= w(-x_0) \zeta(-x_0) - w(-x_0-1) \zeta(-x_0-1) = \int \limits_{-x_0-1}^{-x_0} (w \zeta )' \\
	&= \int \limits_{-x_0-1}^{-x_0} w' \zeta +w \zeta' 
	= \int \limits_{-x_0-1}^{-x_0} w' \sqrt{\psi} \frac{1}{\sqrt{\psi}} \zeta + \int \limits_{-x_0-1}^{-x_0} w \sqrt{\psi} \frac{1}{\sqrt{\psi}} \zeta'\\
	 &\leq \norm {\zeta}_{L^\infty(\R)} \norm { \frac{1}{\sqrt{\psi}}  }_{L^2(-x_0-1,x_0)} \norm {w'}_{L^2((-x_0-1,-x_0), \psi \dx{x})} \\
	 &\quad +\norm {\zeta'}_{L^\infty(\R)} \norm { \frac{1}{\sqrt{\psi}}  }_{L^2(-x_0-1,x_0)} \norm {w}_{L^2((-x_0-1,-x_0), \psi \dx{x})} \\
	 &\leq 2 \norm {\zeta}_{C^1(\R)}  \norm { \frac{1}{\sqrt{\psi}}  }_{L^2(-x_0-1,x_0)} \norm {w}_{H^1((-x_0-1,-x_0), \psi \dx{x})} \\
	 &\leq 14  \norm { \frac{1}{\sqrt{\psi}}  }_{L^2(-x_0-1,x_0)} \norm {w}_{H^1((-x_0-1,-x_0), \psi \dx{x})} 
	\end{split}
	\end{align}
	(The same holds for $-w(0)$).
	With this straight forward estimation we are in the position to do our final embedding in $(-x_0,0)$:

\begin{align}
\int \limits_{-x_0}^0 \abs {w}^3 \psi &\leq \int \limits_{-x_0}^0 \abs {w}^2 \psi \norm {w}_{L^\infty} \leq \int \limits_{-x_0}^0 \abs {w}^2 \psi  \bra { \abs { w(-x_0) } + \int \limits_{-x_0}^0 \abs {w'} } \\
&\leq\int \limits_{-x_0}^0 \abs {w}^2 \psi  \bra { 14 \norm { \frac{1}{\sqrt{\psi}}  }_{L^2(-x_0-1,-x_0)} \norm  {w}_{H^1((-x_0-1,-x_0), \psi \dx{x})}  + \int \limits_{-x_0}^0 \abs {w'}  \sqrt{\psi} \frac{1}{\sqrt{\psi}} } \\
&\leq \norm  {w}^2_{H^1((-x_0,0), \psi \dx{x})}  \left (14 \norm { \frac{1}{\sqrt{\psi}}  }_{L^2(-x_0-1,-x_0)}  \norm  {w}_{H^1((-x_0-1,-x_0), \psi \dx{x})} \right . \\
&\quad + \left . \norm  {w}_{H^1((-x_0,0), \psi \dx{x})}  \norm { \frac{1}{\sqrt{\psi}}  }_{L^2(-x_0,0)} \right ) \\
&\leq \norm  {w}^3_{H^1((-x_0-1,0), \psi \dx{x})}   \max \set { 14 \norm { \frac{1}{\sqrt{\psi}}  }_{L^2(-x_0-1,-x_0)} ,  \norm { \frac{1}{\sqrt{\psi}}  }_{L^2(-x_0,0)}    }  \\
&= \norm  {w}^3_{H^1((-x_0-1,0), \psi \dx{x})} \frac{1}{\sqrt{\psi(-x_0)}} \max \set { 14, \bra { \int \limits_{-x_0}^0 \exp \bra { \int \limits_{-x_0}^t k(s) \dx{s} } \dx{t}   }^\frac{1}{2}      }  
\end{align}
To put all the estimates together we use that for all $b>0$ (arbitrary but fixed) it holds that for all $z \in \R$
\begin{align}
\abs {z}^3 \leq b \abs {z}^2 + \frac{1}{b} \abs {z}^4 .
\end{align}
Hence it follows that for all $b >0$
\begin{align}
\int \limits_{R}^{-x_0} \abs {w}^3 \psi \leq b \int \limits_{R}^{-x_0} \abs {w}^2 \psi + \frac{1}{b} \int \limits_{R}^{-x_0} \abs {w}^4 \psi .
\end{align}
Combining everything we have
\begin{align}
\int \limits_{R}^{0} \abs {w}^3 \psi &= \int \limits_{R}^{-x_0} \abs {w}^3 \psi + \int \limits_{-x_0}^0 \abs {w}^3 \psi \\
&\leq \bra { b \int \limits_{R}^{-x_0} \abs {w}^2 \psi  + \frac{1}{b} \int \limits_{R}^{-x_0} \abs {w}^4 \psi   } + \int \limits_{-x_0}^0 \abs {w}^3 \psi \\
&\leq \bra { b \norm {w}^2_{H^1((R,-x_0), \psi \dx{x})} + \frac{1}{b} \frac{1}{\psi(-x_0)}  \norm {w}^4_{H^1((R,-x_0),\psi \dx{x})}         } \\
&\quad + \frac{1}{\sqrt{\psi(-x_0)}}  \max \set { 14, \bra { \int \limits_{-x_0}^0 \exp \bra { \int \limits_{-x_0}^t k(s) \dx{s} } \dx{t}   }^\frac{1}{2}      }  \norm  {w}^3_{H^1((-x_0-1,0), \psi \dx{x})} \\
&\leq \norm {w}^2_{H^1((R,0),\psi \dx{x})} \left (b + \frac{\norm {w}^2_{H^1((R,0),\psi \dx{x})} }{b \psi(-x_0)} \right .\\
&\quad  \left .+  \frac{1}{\sqrt{\psi(-x_0)}}  \max \set { 14, \bra { \int \limits_{-x_0}^0 \exp \bra { \int \limits_{-x_0}^t k(s) \dx{s} } \dx{t}   }^\frac{1}{2}      }    \norm  {w}_{H^1((-x_0-1,0), \psi \dx{x})}  \right )
\end{align}
Now we choose $b>0$ optimally as
\begin{align}
b:= \frac{\norm {w}_{H^1((R,0),\psi \dx{x})}}{\sqrt{\psi(-x_0)}}  .
\end{align}
It follows then directly that 
\begin{align}
	\int \limits_R^0 \abs {w}^3 \psi \leq \frac{\alpha}{\norm {f''}_{L^\infty}} \norm {w}^2_{H^1((R,0),\psi \dx{x})}
\end{align}
if 
\begin{align}
	\norm {w}_{H^1((R,0),\psi \dx{x})} \leq \frac{\alpha \sqrt{\psi(-x_0)}}{\norm {f''}_{L^\infty} \bra {2 + \max \set { 14, \bra { \int \limits_{-x_0}^0 \exp \bra { \int \limits_{-x_0}^t k(s) \dx{s} } \dx{t}   }^\frac{1}{2}      }}} =: \delta .
\end{align}
But this is exactly the (explicit) choice of $\delta$ and the Lemma is proven.
\end{proof}

With the help of this Lemma we are now in the position to prove the Proposition \ref{proposition:blocking:fixed_R}.
\begin{proof}[Proof of Proposition \ref{proposition:blocking:fixed_R}]
	The strategy of this proof is to show that for all $w \in H^1_{0,1}((R,a), \psi \dx{x})$ such that \begin{align}
	\norm {w-w_0}_{H^1((R,a),\psi \dx{x})} = \delta
	\end{align}
	 it holds that
	\begin{align}
	J_{(R,a)}(w) > J_{(R,a)}(w_0).
	\end{align}
	Since  $J_{(R,a)}$ is weakly lower semicontinous, bounded from below and coercive, $J_{(R,a)}$ has a local minimizer   among all functions $w \in H^1_{0,1}((R,a), \psi \dx{x})$ such that $\norm {w-w_0}_{H^1((R,a),\psi \dx{x})} \leq \delta$. And since we have a local minimizer, that does not lie on the boundary $\norm {w-w_0}_{H^1((R,a),\psi \dx{x})} = \delta $, we can derive an Euler-Lagrange-equation.
	
	Let us first note that from $\norm {w-w_0}_{H^1((R,a), \psi \dx{x})} \leq \delta$ it follows that $\norm {w-w_0}_{H^1((R,0), \psi \dx{x})} \leq \delta$. To make use of Lemma \ref{lemma:auxiliary_lemma_fixed_R} we split the functional as follows. 
	For any $w \in H^1_{0,1}((R,a), \psi \dx{x})$ it holds that
	\begin{align}
	J_{(R,a)}(w) - J_{(R,a)}(w_0 ) = \underbrace{ J_{(R,0)}(w) - J_{(R,0)}(w_0)  }_{=: I}+  \underbrace{J_{(0,a)}(w) - J_{(0,a)}(w_0) }_{=: II}.
	\end{align}
	From Lemma \ref{lemma:auxiliary_lemma_fixed_R} it follows for the first term I 
	\begin{align}
	J_{(R,0)}(w) - J_{(R,0)}(w_0) \geq \alpha \norm {w}^2_{H^1((R,0), \psi \dx{x})} .
	\end{align}
	For the second term II we use the observation that there is $K >0$ such that for all $s\in \R$
	\begin{align}
	F(s) \geq K(s-1)^2 
	\end{align}
	(since $F(1)=0$, $f'(1)<0$, $\int_0^1 f >0$ and $f(s) <0$ for $s \in (0,\theta)$ and $f(s) >0$ for $s \in (\theta,1)$, $f(s)>0$ for $s \in (-\infty,0)$, $f(s)<0$ for $s \in (1,\infty)$ and $f$ grows linearly at infinity.)

	With this observation we can conclude that for any interval $A \subset (R,a)$ it holds that
	\begin{align} \label{J_generic_bound_from_below}
	J_A(w) \geq \nu \norm {w-1}^2_{H^1(A, \psi \dx{x})} ,
	\end{align}
	where $\nu := \min \set {K, \frac{1}{2}}$.
	Furthermore we can estimate 
	\begin{align}
	J_{(0,a)}(w_0) = \int \limits_0^{a} \bra { \frac{1}{2} \bra { \frac{1}{a} }^2 +F(w_0)   } \psi
	\leq \underbrace {  \bra { \frac{1}{2a} +a \max \limits_{s \in [0,1]} F(s)     }  }_{:= \beta}  \psi(0)
	\end{align}
	Together with \eqref{J_generic_bound_from_below} we get
	\begin{align}
	J_{(0,a)}(w) - J_{(0,a)}(w_0) \geq \nu \norm {w-1}^2_{H^1((0,a),\psi \dx{x})} - \beta \psi(0) .
	\end{align}
	In order to use the assumption that $\norm {w-w_0}_{H^1((R,a),\psi \dx{x})} = \delta$, we estimate that 
	\begin{align}
	\nu \norm {w-1}^2_{H^1((0,a), \psi \dx{x})} \geq \frac{\nu}{2} \norm {w-w_0}^2_{H^1((0,a), \psi \dx{x})} - \nu \norm {w_0-1}^2_{H^1((0,a ), \psi \dx{x})} 
	\end{align}
	using Young's inequality.

Furthermore we can explicitly calculate
\begin{align}
\norm {w_0-1}^2_{H^1((0,a), \psi \dx{x})} = \int \limits_0^a  \bra {   \bra  {\frac{1}{a}}^2 + \bra {  \frac{x}{a}-1}^2 } \psi 
=\psi(0)   \underbrace{ \bra {  \frac{1}{a} + \frac{1}{3} a    } }_{:= \gamma} .
\end{align}
Putting these estimations together we get
\begin{align}
J_{(R,a)}(w) - J_{(R,a)}(w_0) &\geq \frac{\nu}{2} \norm {w-w_0}^2_{H^1((0,a), \psi \dx{x})} - \nu \gamma \psi(0) - \beta \psi(0) + \alpha \norm {w-w_0}^2_{H^1((R,0), \psi \dx{x})} \\
&\geq \underbrace { \min \set { \frac{\nu}{2} , \alpha   }  }_{:= \eta}   \underbrace {  \norm { w-w_0  }^2_{H^1((R,a), \psi \dx{x})}  }_{= \delta^2(x_0, \psi,k,f )}- \nu \gamma \psi(0) - \beta \psi(0) .
\end{align}
So the proposition is proved if this is positive. This is the case if
\begin{align}
&\eta \delta^2(x_0, \psi,k,f) - (\nu \gamma +\beta) \psi(0) >0 \Leftrightarrow \\
&\frac{\alpha^2 \eta}{ \norm {f''}^2_{L^\infty}(\nu \gamma +\beta)} > \exp \bra { -\int_\R k  }   \bra {2 + \max \set { 14, \bra { \int \limits_{-x_0}^0 \exp \bra { \int \limits_{-x_0}^t k(s) \dx{s} } \dx{t}   }^\frac{1}{2}      }}^2
\end{align}
Here we have used that 
\begin{align}
	\frac{\psi(0)}{\psi(-x_0)} = \exp \bra { - \int_\R k   }
\end{align}
simply by definition of $\psi$ \eqref{Def:psi}.
This proves the Proposition.
Note that $\beta, \gamma$ only depend on $f$ and $a$ and $\beta, \gamma \rightarrow \infty$ as $a \rightarrow 0$ or $a \rightarrow +\infty$. Since they also depend continuously on $a$ there is an optimal (depending on $f$) choice of $a \in (0,\infty)$.
\end{proof} 

From Proposition \ref{proposition:blocking:fixed_R} we get for any $R<-x_0-1$ existence of a local minimizer $w_R \in H^1_{0,1}((R,a), \psi \dx{x})$ of the functional $J_{(R,a)}$ such that $\norm {w_R-w_0}_{H^1((R,a), \psi \dx{x})} \leq \delta$. From this it follows that $w_R$ is a weak solution of 
\begin{align}
\begin{cases}
-w_R'' + k w_R' = f(w_R) \text{ in } (R,a) \\
w_R(a) =1 \\
w_R(R) =0
\end{cases} .
\end{align} 
Using the maximum principle we conclude that for all $R<-x_0-1: 0 \leq w_R \leq 1$ in $(R,a)$.
From this we construct a supersolution to 
\begin{align}
\begin{cases}
\partial_t u -\partial_{xx} u +k(x) \partial_{x} u= f(u) \text{ for all } (t,x) \in \R\times \R \\
u(t,x) - \phi(x+ct) \rightarrow 0 \text{ as } t \rightarrow -\infty \text{ uniformly in } \R
\end{cases} 
\end{align} 
by passing to the limit $R \rightarrow -\infty$ and extending by $1$ onto $(a,\infty)$.

\begin{prop}[Existence of a stationary supersolution] \label{proposition:existence_stationary_supersolution}
	Assume that condition \eqref{blocking_condition} holds and let $w_R$ be the local minimizer of the energy functional $J_{(R,a)}$ as in Proposition \ref{proposition:blocking:fixed_R}, then $(w_R)_{R<-x_0-1}$ converges up to a subsequence in $C^2_\text{loc}((-\infty,a])$ to a solution $w_\infty$ of
	\begin{align} \label{w_infty_equation}
	\begin{cases}
	- w_\infty''+k w_\infty'  = f(w_\infty) \text{ in } (-\infty,a) \\
	w(a)=1
	\end{cases}
	\end{align}
	such that $w_\infty(x) \rightarrow 0$ as $x \rightarrow -\infty$.
\end{prop}

\begin{proof}
	As $0\leq w_R \leq 1$ for all $R>-x_0-1$ and using Schauder estimates there exists a subsequence $(R_n)_{n \in \N}$ with $R_n \searrow -\infty$ as $n \rightarrow \infty$ such that $w_{R_n} \rightarrow w_\infty$ in $C^2_\text{loc}((-\infty,a])$ as $n \rightarrow \infty$.
	It remains to prove that the limit $w_\infty$ satisfies $w_\infty \rightarrow 0$ as $x \rightarrow -\infty$. By Fatou's Lemma we find
	\begin{align} \label{Fatou_w_infty_w_0_bound}
	\norm {w_0-w_\infty}^2_{H^1((-\infty,a), \psi \dx{x})} \leq \liminf \limits_{n \rightarrow \infty} \norm {(w_0-w_{R_n})  \chi_{\set {R_n<x<a}}   }^2_{H^1((-\infty,a), \psi \dx{x})} \leq \delta^2.
	\end{align}
	
	Then arguing by contradiction, we assume that there exists $\eta >0$ and a sequence $(x_n)_{n \in \N}$, such that $x_n \rightarrow -\infty$ as $n \rightarrow \infty$ and $w_\infty(x_n) >\eta$ for all $n \in \N$. Since $w_\infty \in C^2_\text{loc}$ and $w_\infty$ is a bounded solution of \eqref{w_infty_equation}, by standard elliptic estimates we know that $\abs { w'_\infty} \leq K$ for some constant $K>0$. It follows that for all $x \in \bra { x_n -\frac{\eta}{2K} , x_n + \frac{\eta}{2K} }$
	\begin{align}
	\abs {w_\infty(x) - w_\infty(x_n)} \leq \max \limits_{\pra { x_n -\frac{\eta}{2K} , x_n + \frac{\eta}{2K} }  } \abs {w'_\infty} \abs {x-x_n}
	\end{align}
	Hence $w_\infty(x) \geq \frac{\eta}{2}$ for all $x \in \bra { x_n -\frac{\eta}{2K} , x_n + \frac{\eta}{2K} }$ and all $n \in \N$. Assuming without loss of generality that for all $n \in \N$ it holds that $x_{n+1} < x_n -\frac{K}{\eta}$, we obtain that
	\begin{align}
	\norm {w_\infty -w_0}^2_{L^2((-\infty,a), \psi \dx{x})} \geq \psi(-x_0) \sum \limits_{n =M}^\infty \bra {\frac{\eta}{2} }^2 \frac{\eta}{K} = \infty,
	\end{align}
	where $M \in \N$ is such that $x_n <-x_0$ for all $n >M$. But this is a contradiction to \eqref{Fatou_w_infty_w_0_bound} and thereby the Proposition is proved.
\end{proof}

The proof of Theorem \ref{theorem:blocking} is now nothing but applying Proposition \ref{proposition:existence_stationary_supersolution} and a comparison principle.

\begin{proof}[Proof of Theorem \ref{theorem:blocking}]

	Let us now take $w_\infty$ as in Proposition \ref{proposition:existence_stationary_supersolution} and let us  extend $w_\infty$ by $1$ onto all of $\R$. We set
	\begin{align}
	\tilde{w}_\infty(x) := 
	\begin{cases}
	w_\infty(x) &, \text{ if } x \leq a \\
	1 &, \text{ else}.
	\end{cases}
	\end{align}
	Thus $\tilde{w}_\infty(x) $ is a supersolution of the parabolic problem 
	\begin{align}
	\partial_t u - \partial_{xx} u + k(x) \partial_{x} u  = f(u) &\text{ for } (t,x) \in \R \times \R , 
	\end{align}
	Furthermore it holds that
	\begin{align}
	\lim \limits_{t \rightarrow -\infty} \inf \limits_{x \in \R}  \bra {  \tilde{w}_\infty(x) - \phi(x+ct) } \geq 0.
	\end{align}
	Indeed 
	\begin{align}
	&\text{ for } x \geq a, \text{ for all } t \in \R:  \tilde{w}_\infty(x) - \phi(x+ct) = 1- \phi(x+ct) \geq 0 \\
	&\text{ for } x < a, \text{ for all } t \in \R:  \tilde{w}_\infty(x) - \phi(x+ct) \geq \tilde{w}_\infty(x) - \phi(a+ct) \rightarrow \tilde{w}_\infty(x) \geq 0 \\
	&\qquad \text{ uniformly in } (-\infty,a) \text{ as } t \rightarrow -\infty.
	\end{align}
	Using the generalized maximum principle (Lemma 3.2 in \cite{front_blocking}), we conclude that
	\begin{align}
	u(t,x) \leq \tilde{w}_\infty(x) \text{ for all } t \in \R \text{ and } x \in \R.
	\end{align} 
	Hence the stationary supersolution $\tilde{w}_\infty$ blocks the invasion of the stationary state $1$ into the left.
	
\end{proof}

\section{A sufficient condition for the wave passing the drift disturbance with asymptotically at most a shift} \label{section:existence/propagation}

The objective of this section is to give a sufficient condition for `mostly unchanged' propagation as it is given in Theorem \ref{main_theorem}.

The strategy is the well-known one of \cite{FifeMcLeod}, i.e. constructing suitable super-and subsolutions that will ensure invasion of the front to the left and imply suitable a priori estimates for a Lyapunov-function argument.
To do so first of all we need to make sure that we get full invation by bounding $u$ from below against a slightly disturbed traveling wave in Lemma \ref{Lemma:apriori_lower_estimate}. Then we make sure that besides the drift-disturbance the solution $u$ looks approximately like a traveling wave even after the disturbance. Therefore we show that for all large times $u$ is small for $x$ sufficiently negative. (see Lemma \ref{lemma:u_small_x_1_small}).
Having established this we construct a priori estimates from above and below as in Lemma \ref{Lemma_estimation_above} and establish the stability Lemma \ref{lemma_stability} that we need for our Lyapunov function argument.

\begin{lem}[estimation from below / full invasion] \label{Lemma:apriori_lower_estimate}
	Under the assumptions of Theorem \ref{main_theorem} there exists a time $T^-\in \R$ and constants $\beta^-, C^-, \omega^- >0$ such that 
	\begin{align}
	\max \{	\phi(x+ct - \beta^-) -C^- e^{-\omega^- t},0 \} \leq u(t,x)
	\end{align}
	for $x \in \R$ and $t \geq T^-$ .
\end{lem}

\begin{proof}
	From the initial condition in \eqref{DiffEqu} we know that for every $\epsilon >0 $ there is $t_\epsilon \in \R$ such that
	\begin{align}
	\norm { u(t_\epsilon, \cdot) - \phi((\cdot)+ct_\epsilon)  }_{L^\infty(\R)} \leq \epsilon .
	\end{align}
Note that $t_\epsilon$ only depends on $f$ and $\epsilon$ since
\begin{align}
	0 \leftarrow w^-(t,x) -\phi(x+ct) \leq u(t,x) -\phi(x+ct) \leq w^+(t,x)- \phi(x+ct) \rightarrow 0 
\end{align}
as $t \rightarrow -\infty$ uniformly in $\R$.
Since $w^ -$ and $w^ +$ (see section \ref{section:existence_and_uniqueness}) depend only on $f$ it is clear that $t_\epsilon$  depends only on $f$ and $\epsilon$ as claimed.

It is also known (cf. \cite{FifeMcLeod}) that there is $C_\phi>0$ and $\mu >0$ such that
\begin{align} \label{estimate:phi_prime}
0 < \phi'(z) < C_\phi e^{-\mu |z|} \text{ for all } z \in \R.
\end{align}
	Let us set furthermore
\begin{align} \label{def:omega-}
\omega^- := \min \left ( \frac{|f'(0)|}{4}, \frac{|f'(1)|}{4}, \frac{c \mu}{2}  , 1      \right ) 
\end{align}
and choose $\rho >0$ such that
\begin{align}
\begin{cases}
|f'(s) -f'(0)| \leq \omega^-  &\text{ for all } s \in [0,\rho] \\
|f'(s) -f'(1)| \leq \omega^- &\text{ for all } s \in [1-\rho,1]
\end{cases} \quad .
\end{align}
Let $A^->0$ be such that 
\begin{align}
\begin{cases}
\phi(z) \geq 1-\frac{\rho}{2} &\text{ for all } z \geq A^-, \\
\phi(z) \leq \frac{\rho}{2} &\text{ for all } z \leq -A^- .
\end{cases}
\end{align}
Since $\phi'$ is positive and continuous on $\R$ we have
\begin{align}
\delta^- &:= \min \limits_{z \in [-A^-, A^-]} \phi'(z) >0. \label{def:delta-} 
\end{align}
We choose $\epsilon =\frac{\rho}{4}$ and set 
\begin{align}
	m := \frac{\norm {f'}_{L^ \infty}+2 \omega^-}{\delta^-} \bra { \frac{\epsilon}{\omega^- } + \frac{C_\phi e^{\mu x_0} e^{-\mu ct_\epsilon}  }{\mu c \omega^- }    } >0.	
\end{align}
Let us define the following auxiliary functions $v^-$ by
\begin{align} \label{def:v-}
	\dot{v}^-(t) &= -\omega^- v^-(t) +k^+ C_\phi  e^{-\mu (-x_0-m+ct)} \text{ for all } t >t_\epsilon \text{ and } \\
	v^-(t_\epsilon) &= \epsilon
\end{align}
as well as
	\begin{align} \label{def:V-}
V^-(t) &:=  \frac{\|f'\|_{L^\infty}+2\omega^-}{\delta^-}   \int \limits_{t_\epsilon}^t v^-(\tau) \dx{\tau} .
\end{align}
This implies that
\begin{align}
v^-(t) = \bra {  \epsilon + \underbrace{ \frac{C_\phi k^+ e^{\mu (x_0+m)} e^{-\mu ct_\epsilon}}{c \mu - \omega^-}  }_{=: C_{v^-}  } } e^{-\omega^- (t-t_\epsilon)    } - \frac{C_\phi k^+ e^{\mu (x_0+m)} e^{-\mu ct_\epsilon}}{c \mu - \omega^-}  e^{-\mu c(t-t_\epsilon)}
\end{align}
It directly follows that
\begin{align}
	v^-(t) &\geq \epsilon e^{-\omega^-(t-t_\epsilon)} >0 \quad \text{ (from $\omega^- < c\mu $) and } \\
	v^-(t) &\leq \epsilon + \frac{k^+ C_\phi}{\mu c} e^{\mu(x_0+m)}  e^{-\mu ct_\epsilon} \quad \text{(from $v^- >0$ and \eqref{def:v-})}
\end{align}
Under the condition 
\begin{align} \label{condition_k+}
k^ + < \min \set { \frac{\rho}{4} \frac{\mu c}{C_\phi} e^ {-\mu (x_0 +m)} e^{-\mu ct_\epsilon} , e^{-\mu m} }
\end{align}
it holds that
\begin{align} \label{bounds_on_v-_and_V-}
0<	v^-(t) \leq \frac{\rho}{2} \text{ and } 0 \leq V^-(t) \leq V^-(\infty) \leq m \quad \text{ for all } t \geq t_\epsilon .
\end{align}
(Note that \eqref{condition_k+} can be written in the form of condition \eqref{condition_theorem_propagation} since the right hand side in \eqref{condition_k+} depends only on $f$ and $x_0$. Recall that $\epsilon$ does only depend on $f$ and $t_\epsilon$ does only depend on $f$ and $\epsilon$.) We can construct a subsolution as follows.
Set
\begin{align}
	u^-(t) := \phi(\xi^-(x,t)) -v^-(t) \quad \text{ for all } t \geq t_\epsilon \text{ and } x \in \R,
\end{align}
where $\xi^-(t,x) := x+ct -V^-(t)$ are perturbed moving-frame coordinates.
In order to show that $u^-$ is a subsolution, we need to prove that 
\begin{align}
	\sL u^-:= \partial_t u^- -\partial_{xx} u^- +k(x) \partial_{x} u^- -f(u^-) \leq 0 \quad \text{ for all } t > t_\epsilon, x \in \R .
\end{align}
Since $u \geq 0$ and $0$ is a trivial subsolution of $\sL$ and maxima of subsolutions are again subsolutions, it sufices to show that $\sL u^- \leq 0$ on $\set { u^- >0   }$.

From the definition of $u^-$ and the auxiliary functions it follows directly
\begin{align}
	\sL u^- &= \bra { c- \dot{V}^-(t)   } \phi'(\xi^-) -\dot{v}^-(t) - \phi''(\xi^-) +k(x) \phi'(\xi^-) -f  \bra { \phi(\xi^-) -v^-(t)  } \\
	&= - \dot{V}^-(t) \phi'(\xi^-) -\dot{v}^-(t)  +k(x) \phi'(\xi^-)  +f \bra {\phi(\xi^-)}  -f  \bra { \phi(\xi^-) -v^-(t)  } \\
	&\leq -\dot{V}^-(t) \phi'(\xi^-) -\dot{v}^-(t) +k^+ C_\phi e^{ -\mu (-x_0-m +ct)  }  +f \bra {\phi(\xi^-)}  -f  \bra { \phi(\xi^-) -v^-(t)  } \\
	&=  -\dot{V}^-(t)  \phi'(\xi^-) + \omega^- v^-(t) + f \bra { \phi(\xi^-)  } -f \bra { \phi(\xi^-) -v^-(t)   }    ,
\end{align} 
we have used  \eqref{estimate:phi_prime}, \eqref{def:v-}, \eqref{def:V-}, \eqref{bounds_on_v-_and_V-}, and \eqref{ODE_travelling_wave}.
We now make the usual distinction between the cases $\xi^- <-A^-$, $\xi^- \in [-A^-, A^-]$ and $\xi^- > A^-$.

If $\xi^- <-A^-$ (using $\dot{V}^ - \geq 0$ and $\phi' >0$) we obtain that
\begin{align}
\sL u^- &\leq \omega^- v^-(t) + f'(\sigma) v^-(t) \leq \omega^- v^-(t) + \bra { f'(0) +\omega^-    } v^-(t) \\
&= \bra {  f'(0)  +2 \omega^-  } v^-(t)  \leq 0 ,
\end{align}
where $\sigma \in \bra { 0, \frac{\rho}{2}   } $ comes from the mean value theorem and the last estimate comes from the choice of $\omega^-$ in \eqref{def:omega-}. 

If $ \xi^- > A^-$ (using $\dot{V}^ - \geq 0$, $\phi'>0$ and in this case $ \phi( \xi^-) > 1- \frac{\rho}{2}$, $ \phi(\xi^-) -v^-(t) > 1-\frac{\rho}{2}  -\frac{\rho}{2} = 1-\rho   $) it holds that
\begin{align}
	\sL u^- &\leq \omega^- v^-(t) + f'(\tilde{\sigma}) v^-(t) \leq \omega^- v^-(t) + \bra { f'(1) +\omega^-       } v^-(t) \\
	&= \bra { f'(1) +2 \omega^-   } v^-(t) \leq 0,
\end{align}
where $\tilde{\sigma} \in \bra { 1-\rho, 1  }$ comes again from the mean value theorem and the last estimate follows from the choice of $\omega^-$.

If $\xi^- \in [ -A^-,A^-  ]$ we have that $ \phi' > \delta^-$ and $ \dot{V}^ - >0$,
\begin{align}
	\sL u^- &\leq -\dot{V}^-(t) \delta^- + \omega^- v^-(t) + \norm { f'  }_{L^\infty} v^-(t) \\ &
	= - \frac{\norm { f'  }_{ L^\infty  }  + 2 \omega^- }{\delta^-} \delta^- v^-(t) + \omega^- v^-(t) + \norm { f'  }_{L^\infty} v^-(t) = - \omega^- v^-(t) \leq 0 ,
\end{align} 
where we have used \eqref{def:delta-} and \eqref{def:V-}.

This was to be proven.

\end{proof}

Having this bound from below, that ensures that the wave fully invades the left, we just have to wait until only the tail, where the solution $u \approx 1$, lies in the support of $k$ and to ensure that $u$ is close to zero far to the left to get a similar bound from above against a slightly perturbed traveling wave.
For this we need the following lemma.

\begin{lem}[$u$ is small for $x$ small] \label{lemma:u_small_x_1_small}
Under the assumptions of Theorem \ref{main_theorem} for every $\epsilon >0$ there is $T^\epsilon \in \R$ such that for every $T>T^\epsilon$ there is $\zeta^\epsilon(T)<-x_0$ such that 
\begin{align}
	u(T,x) \leq \epsilon \quad \text{ for all } x \leq \zeta^\epsilon(T) .
\end{align}
\end{lem}

\begin{proof}
	The idea of the proof i.e. constructing supersolutions by slightly increasing the nonlinearity around the stable states $0$ and $1$ is taken from \cite{MatanoObst}.
	
	Choose $0<\delta< \min \set {\theta,\epsilon}$ arbitrarily ($\theta$ is as in \eqref{F_4}.). Extend $f$ linearly to a $C^{1,1}$ function on $[0,1+\delta]$.
	Let us now take $f^\delta \in C^ 2([0,1+\delta])$ such that $f^\delta  \bra {\frac{\delta}{2}} = 0 = f^\delta \bra {   1+ \frac{\delta}{2}   } $, $ {f^\delta}' \bra {\frac{\delta}{2}} <0, {f^\delta}' \bra {1+ \frac{\delta}{2}} <0 $ and $f^\delta \geq f$. 
		Then there is a unique $c^\delta >c>0$ and $\phi^\delta \in C^2(\R)$ (unique up to translation) such that
		\begin{align}
			{\phi^\delta}'' -c^\delta {\phi^\delta}' + f^\delta \bra {\phi^\delta} = 0 \text{ in } \R \\
			\text{ and } \frac{\delta}{2} = \phi^\delta(-\infty) < \phi^\delta < \phi^\delta(\infty) =1 + \frac{\delta}{2} \text{ and } {\phi^\delta}' >0 \text{ in } \R.
		\end{align}
		By the initial condition in \eqref{DiffEqu} there is $t^\delta \in \R$ such that 
		\begin{align} \label{auxiliary_lemma_delta:initial_condition1}
			\abs { \phi \bra {x+c t^\delta} -u \bra { t^\delta,x  }  } \leq \frac{\delta}{4} \text{ for all } x \in \R .  
		\end{align}
		Now we chose $\zeta^\delta <-x_0$ such that 
		\begin{align} \label{auxiliary_lemma_delta:initial_condition2}
			\phi \bra {x+c t^\delta} \leq \frac{\delta}{4} \text{ for all } x \leq \zeta^\delta.
		\end{align}
		Then there is $m^\delta \in \R$ such that 
		\begin{align}
			\phi^\delta \bra { \zeta^\delta +c^\delta t^\delta +m^\delta  } \geq 1.
		\end{align} 
		This implies that
		\begin{align} 
		\phi^\delta \bra { \zeta^\delta +c^\delta t +m^\delta     } \geq 1 \geq u \bra {t, \zeta^\delta} \text{ for all } t \geq t^\delta, 
		\end{align}
		since ${\phi^\delta}'>0$.
		From \eqref{auxiliary_lemma_delta:initial_condition1} and \eqref{auxiliary_lemma_delta:initial_condition2} it follows 
		\begin{align}
			u \bra { t^\delta, x  }  \leq \phi \bra {x +c t^\delta} + \frac{\delta}{4}  \leq \frac{\delta}{2} \leq \phi^\delta \bra {x +c t^\delta + m^\delta} \text{ for all } x \leq \zeta^\delta.
		\end{align}
		Furthermore $\phi^\delta \bra {x +c^\delta t + m^\delta    }$ solves
		\begin{align}
			\partial_t \phi^\delta - \partial_{xx} \phi^\delta = f^\delta \bra { \phi^\delta} &\text{ in } \R \times\set {x \leq \zeta^\delta }  , 
		\end{align}
		Hence from the parabolic comparison principle it follows that 
		\begin{align}
			u(t,x) \leq \phi^\delta \bra { x + c^\delta  t + m^\delta   } \text{ for all  } t \geq t^\delta \text{ and } x \leq \zeta^\delta .
		\end{align}
		From this we see that for every $T> t^\delta$ there is $\zeta^\epsilon(T) \leq \zeta^\delta$ such that
		\begin{align}
			u(T,x) \leq \phi^\delta \bra { x +c^\delta T + m^\delta   } \leq \epsilon \text{ for all } x \leq \zeta^\epsilon(T) .
		\end{align}
		This concludes the proof.
\end{proof}

With this Lemma we are now in the position to prove the following upper bound on the solution $u$.

\begin{lem}[estimation from above] \label{Lemma_estimation_above}
Under the assumptions of Theorem \ref{main_theorem} there is a time $T^+ >T^-$ and constants $\beta^+,C^+,\omega^+ >0$ such that 
\begin{align}
	\min \set { \phi \bra { x + c t + \beta^+      } +C^+ e ^{-\omega^+ t   }  ,1    } \geq u(t,x) \quad \text{ for all } x \in \R, t \geq T^+.
\end{align}
\end{lem}

\begin{proof}
	The proof is similar to the one of Lemma 4.4 in \cite{change_of_speed_1}.	
	We choose 
	\begin{align} \label{def:omega+}
		\omega^+ := \min \set { \frac{|f'(0)|}{4} ,  \frac{|f'(1)|}{4} , \frac{\mu c}{2} ,1   } ,
	\end{align}
	where $\mu$ is as in \eqref{estimate:phi_prime}.
	Let us choose $\rho >0$ such that 
	\begin{align}
		\begin{cases}
		\abs { f'(s) -f'(0)   } < \omega^+ &\quad \text{ for all } s \in [0,\rho] \\
		\abs { f'(s) -f'(1)   } < \omega^+ &\quad \text{ for all } s \in [1-\rho,1] .
		\end{cases}
	\end{align}
	Let $A^+ >0$ be such that 
	\begin{align} \label{def:A+}
		\begin{cases}
		\phi(z) \geq 1- \frac{\rho}{2} &\quad \text{ for all } z \geq A^+ \\
		\phi(z) \leq \frac{\rho}{2} &\quad \text{ for all } z \leq -A^+ .
		\end{cases}
	\end{align}
	Let $\gamma = \frac{\rho}{4}$ be arbitrary and let $T>0$ be large enough such that 
	\begin{align}
	T>\max \set {T^-,T^\gamma,0} ,
		A^+ -cT \leq -x_0, T \geq \frac{x_0}{c} \text{ and } \frac{\norm {k}_{L^\infty} C_\phi e^{ \mu x_0 }  }{c \mu} e^{-c\mu T} < \frac{\rho}{4} .
	\end{align}We understand $T^-$ and $T^\delta$ as in Lemmata \ref{Lemma:apriori_lower_estimate} and \ref{lemma:u_small_x_1_small}.
	Let  $\tilde{\zeta}_- = \zeta^\gamma(T) < -x_0 $ be such that 
	\begin{align}
		u(T,x) \leq \gamma \quad \text{ for all } x \leq \tilde{\zeta}_- .
	\end{align}
	Here we understand $\zeta^\gamma(T)$ as in Lemma \ref{lemma:u_small_x_1_small}.
	Let furthermore $\tilde{\zeta}_+>0$ be  such that 
	\begin{align}
		\phi \bra { x +cT -\beta^-   } \geq 1-\gamma  \text{ and hence also } \\
		\phi \bra { x +cT    } \geq 1 -\gamma \text{ for all } x \geq \tilde{\zeta}_+ .
	\end{align}
	Let us set $D_\gamma := \set { x \in \R : \tzeta_- \leq x \leq \tzeta_+  }$.
	Since $\lim \limits_{z \rightarrow \infty} \phi(z)=1$ and $\max \limits_{x \in \overline{D_\gamma}} u(T,x) <1$, there is $\beta >0$ such that
	\begin{align} \label{definition:beta}
		\phi (x +cT + \beta) \geq u(T,x) \quad \text{ in } D_\gamma.
	\end{align} 
	Finally let us define
	\begin{align}
	 \delta^+ := \min \limits_{z \in [-A^+,A^+]} \phi'(z)>0 ,C_{v^+} := \norm {k}_{L^\infty} C_\phi e^{\mu x_0}  >0
	 \end{align}
	  and for all $t \geq T$:
	\begin{align}
		v^+(t) &:= \bra { \gamma + \frac{C_{v^+} e^{-c \mu T}}{c \mu - \omega^+}    } e^{-\omega^+(t-T)} - \frac{C_{v^+} e^{-c \mu T}}{c \mu - \omega^+}  e^{-c \mu (t-T)}, \\
		V^+(t) &:= \frac{\norm {f'}_{L^\infty} + \omega^+}{\delta^+} \int \limits_T^t v^+(\tau ) \dx{\tau} \geq 0. \label{def:V+}
	\end{align}
	For later reference let us note that
	\begin{align} \label{properties_v_+}
	\begin{split}
		0< \gamma e^{-\omega^+ (t-T)} \leq v^+(t) \leq \frac{\rho}{2} \text{ for all } t \geq T , v^+(T) = \gamma \text{ and } \\
		\dot{v}^+(t) = - \omega^+ v^+(t) + C_{v^+} e^{-c \mu t} \text{ for all } t \geq T.
		\end{split}
	\end{align}
	Now we have everything in place to define our candidate for the supersolution
	\begin{align}
		u^+(t,x) := \min \set { \phi(\xi^+) + v^+(t) , 1     } ,
	\end{align}
	where $\xi^+(t,x) := x +ct + \beta +V^+(t)$. 
	For the initial time $T$ we have that either $x \in D \setminus D_\gamma$ and then 
	\begin{align}
		u(T,x) \leq \min \set { \phi(x+cT ) + \gamma, 1   } = \min \set { \phi(x +cT) +v^+(T),1   } \leq u^+(T,x)
	\end{align}
	or we are in the case $x \in D_\gamma$ then by choice of $\beta >0$ in \eqref{definition:beta} it holds that
	\begin{align}
		u(T,x) \leq \phi(x +cT +\beta) = \phi(\xi^+(T,x)) \leq u^+(T,x) .
	\end{align}
	It remains to show that $u^+$ is indeed a supersolution of the operator $\sL$, i.e that
	\begin{align}
		\sL u^+ = \partial_t u^+ - \partial_{xx} u^+ +k(x) \partial_{x} u^+ -f(u^+) \geq 0 \text{ in } \R \text{ for all } t \geq T.
	\end{align}
	We can estimate this as
	\begin{align}
		\sL u^+ &= \dot{\xi}^+ \phi'(\xi^+) + \dot{v}^+(t) -\phi''(\xi^+) +k(x) \phi'(\xi^+) -f(u^+) \\
		&= \dot{V}^+(t) \phi'(\xi^+) + \dot{v}^+(t) +k(x) \phi'(\xi^+) +f(\phi(\xi^+)) -f(\phi(\xi^+)+v^+(t)) \\
		&\geq 
		\begin{cases}
		\dot{V}^+(t) \phi'(\xi^+) + \dot{v}^+(t)  +f(\phi(\xi^+)) -f(\phi(\xi^+)+v^+(t)) &, x <-x_0 \\
		\dot{V}^+(t) \phi'(\xi^+) + \dot{v}^+(t) -\norm {k}_{L^\infty} C_\phi e^{\mu x_0} e^{-c \mu t} +f(\phi(\xi^+)) -f(\phi(\xi^+)+v^+(t)) &, x \geq -x_0 .
		\end{cases}
	\end{align}
	Let us distinguish as in the proof of Lemma \ref{Lemma:apriori_lower_estimate} between the cases $\xi^+ > A^+$, $\xi^+ \in [-A^+,A^+]$, $\xi^+ <-A^+$.
	In the case $\xi^+ > A^+$ we have to distinguish between $x \geq -x_0$ and $x \leq -x_0$. 
	Using \eqref{properties_v_+} we get in the case $x \geq -x_0$
	\begin{align}
		\sL u^+ &\geq \dot{V}^+(t) \phi'(\xi^+) - \omega^+ v^+(t) + \norm {k}_{L^\infty} C_\phi e^{\mu x_0} e^{-c \mu t} -\norm {k}_{L^\infty} C_\phi e^{\mu x_0} e^{-c \mu t} -f'(\sigma) v^+(t) \\
		&\geq - (f'(\sigma) +\omega^+) v^+(t) \geq 0 ,
	\end{align}
	where $\sigma \in \bra { 1- \frac{\rho}{2}, 1  }$ comes from the mean value theorem and we have used the definition of $A^+$ and $\omega^+$ and that $\dot{V}^ + \geq 0$ and $\phi'>0$ (see \eqref{ODE_travelling_wave}, \eqref{def:omega+}, \eqref{def:A+}, \eqref{def:V+}).
	If $\xi^+ > A^+$ and $x < -x_0$ we can estimate 
	\begin{align}
		\sL u^+ &\geq \dot{V}^+(t) \phi'(\xi^+) -\omega^+ v^+(t) +C_{v^+} e^{-c\mu t} -f'(\sigma ) v^+(t) \\
		&\geq - (f'(\sigma )  + \omega^+) v^+(t) \geq 0, 
	\end{align}
	where again $\sigma \in \bra { 1- \frac{\rho}{2}, 1  }$ comes from the mean value theorem and we have used again the  definition of $A^+$ and $\omega^+$ and that $\dot{V}^ + \geq 0$ and $\phi'>0$.
	In the case $\xi^+ \in [-A^+,A^+]$ we are by choice of $T$ always in the portion of $\R$ where $x \leq -x_0$ and hence we can estimate
	\begin{align}
		\sL u^+ &\geq \dot{V}^+(t) \phi'(\xi^+) -\omega^+ v^+(t) +C_{v^+} - \norm {f'}_{L^\infty} v^+(t) \\
		&\geq \dot{V}^+(t) \delta^+ - (\omega^+ + \norm {f'}_{L^\infty}) v^+(t) \geq 0 
	\end{align}
	by choice of $V^+$.
	In the case $\xi^+ <-A^+$ we are by choice of $T$ always in $ \set {x \leq -x_0}$ and hence the estimation is done as in the case $\xi^+ > A+$ and $x \leq -x_0$.
	
	This concludes the proof.
	
\end{proof}

Having established these a priori estimates, the proof of the long time behaviour is classical and follows the lines of \cite{FifeMcLeod} (in the one dimensional case) or \cite{change_of_speed_1} (in higher dimensions).
First of all we derive the following global estimates on the derivatives of the solution.

\begin{lem}[stability] \label{lemma_stability}
	Let $u$ be a solution of \eqref{DiffEqu} that is at a time $t_0 >T^+$ already close to a traveling wave $\phi(x+ct+\beta)$ for some $\beta \in \R$ i.e. 
	\begin{align}
	|u(t_0,x) -\phi(x+ct_0 + \beta) | \leq \epsilon \text{ for all } x \in \R,
	\end{align}
	where $0 < \epsilon < \frac{\rho}{2}$. Then it holds for all $t \geq t_0$ and $x \in \R$ that
	\begin{align}
	|u(t,x)-\phi (x+ct+\beta)| \leq \delta(\epsilon,t_0) ,
	\end{align}
	where $\delta(\epsilon, t_0) \searrow 0$ as $\epsilon \searrow 0$ and $t_0 \nearrow +\infty$, $T^+$ and $\rho$ are as in the proof of Lemma \ref{Lemma_estimation_above}.
	
\end{lem}

\begin{proof}
	The proof is as in \cite{change_of_speed_1} and shall only be repeated for the sake of self-contain\-ment.
	Note that unlike in the stability result in \cite{FifeMcLeod}, it is not sufficient for the solution of \eqref{DiffEqu} to once be close to a traveling wave in order to remain as close indefinitely. The reason is that the tail of the wave will always lie in a region where $k \neq 0$ and will therefore introduce a disturbance that enters in the form of a possible shift. But since this possible shift is integrable, we can make sure that we do not get driven too far from $\phi(x+ct + \beta) $ if we start late enough and thereby do not accumulate too much of the disturbance. 

Let us now turn to the formalities of the proof. It consists of revisiting Lemma \ref{Lemma_estimation_above}. We take $\gamma$ in the proof of Lemma \ref{Lemma_estimation_above} equal to $\epsilon$. If $t_0$ is large enough such that only the tail of $\phi(x+c+\beta)$ lies right of $x=-x_0$, i.e. $T<t_0$, we know that
\begin{align}
u(t,x) \leq \phi(x+c t + \beta + V^+(t)) + v^+(t)
\end{align}	
where 
\begin{align}
v^+(t) = (\epsilon-C(t_0)) e^{-\omega^+(t-t_0)} + C(t_0) e^{-\mu c(t-t_0)} \leq \epsilon +2 C(t_0)
\end{align}
and $C(t_0) \searrow 0$ as $t_0 \nearrow + \infty$. This implies that
\begin{align}
V^+(t) = C \int \limits_{t_0}^t v^+(\tau) \dx{\tau} \leq C \left (  \frac{\epsilon}{\omega^+} +C(t_0) \left ( \frac{1}{\omega^+} + \frac{1}{\mu c}   \right )   \right) = C (\epsilon+ C(t_0)).
\end{align}
Therefore we know that 
\begin{align}
u(t,x) - \phi(x+c+\beta) &\leq \phi(x+c t +\beta +V^+(t)) - \phi(x+c t+\beta) +v^+(t)  \\
&\leq \| \phi'   \|_{L^\infty} V^+(t) + v^+(t) = C(\epsilon + C(t_0)).
\end{align}
Along the same lines of the proof of Lemma \ref{Lemma_estimation_above} we can get symmetric estimates from below.
So we get that
\begin{align}
| u(t,x) - \phi(x+c t + \beta)  | \leq C ( \epsilon + C(t_0)) .
\end{align}
Since $C(t_0) \searrow 0$ for $t_0 \nearrow + \infty$ this was to be proven.

\end{proof}

From here on it will be more convenient to work in moving frame coordinates $(z,y)$ where $z = x +ct$. In the new coordinates $u$ solves
\begin{align}
\partial_t u + (c+k(z-ct)) \partial_z u - \partial_{zz} u = f(u) \quad &\text{ in } \R \times \R, 
\end{align}

\begin{lem} \label{apriori_bounds_u}
	There is $\sigma >0$ with $\sigma > \frac{|c|}{2}$, $\omega >0$ and $C>0$ such that
	\begin{align}
	|1-u|, |\partial_{z} u|, |\partial_{zz} u|, |\partial_t u | < C (e^{(\frac{1}{2}c-\sigma)z} + e^{-\omega t}) &\quad ,z>0 \\
	|u|, |\partial_{z} u|, |\partial_{zz} u|, |\partial_t u | < C (e^{(\frac{1}{2}c+\sigma)z} + e^{-\omega t}) &\quad ,z<0 
	\end{align}
	(where we always have omitted the arguments $(t,z)$.)
\end{lem}

\begin{proof}	
	We are following the proof of Lemma 4.3 in \cite{FifeMcLeod}.
	It is well known (and can be seen by linearizations around $1$ and $0$) that the wave-front $\phi$ approaches $1$ and $0$ exponentially.
	E.g. the linearisation around  $\phi =1$ shows that $\phi(z) \rightarrow 1$ for $z \rightarrow +\infty$ with approximately the rate
	\begin{align}
	\exp \left ( \frac{1}{2} \left ( c - \sqrt{c^2-4 f'(1)} \right )z \right)
	\end{align}
	For $ z \rightarrow - \infty$ one gets a similar result.
	Together with Lemmata \ref{Lemma:apriori_lower_estimate} and \ref{Lemma_estimation_above} and setting $\omega := \min \set {\omega^-,\omega^+}$ we find:
	\begin{align} \label{bound_zeroth_order}
	\begin{split}
	|u(t,z) | &\leq \phi(z+\beta^+) +C^+ e^{-\omega t}  \\
	&\leq C \left ( \exp \left ( \left ( \frac{1}{2} c + \sigma \right ) z \right ) + e^{-\omega t} \right ) \text{ for } z <0 \text{ and } \\
	|1- u(t,z) | &\leq 1 - \left (   \phi  ( z-\beta^- ) - C^- e^{-\omega t}         \right ) \\
	&\leq C \left ( \exp \left ( \left ( \frac{1}{2} c - \sigma \right ) z \right ) + e^{-\omega t} \right ) \text{ for } z >0
	\end{split}
	\end{align}
	
	Since $f$ is Lipschitz, there is $L>0$ such that 
	\begin{align}
	|f(u)| \leq L |u| \text{ and } |f(u) | \leq L |1-u| \text{ for } u \in [0,1].
	\end{align}
	This together with \eqref{bound_zeroth_order} implies
	\begin{align}
	|f(u(t,z))| \leq C \left (   \exp \left (    \frac{1}{2} c z - \sigma |z|                                       \right )      +e^{-\omega t}                                  \right ) .
	\end{align}
	For the higher order estimates we employ Schauder Theory (e.g \cite{friedman2008partial} Thm 5 Chap 3 and Thm 4 in Chap 7 for the a priori bound on the Hölder-norm of $f(u)$).  Hence it does also hold:
	\begin{align}
	|\partial_{z} u | , |\partial_{zz} u | , |\partial_t u| \leq C \left (   \exp \left (    \frac{1}{2} c z - \sigma |z|                                       \right )      +e^{-\omega t}                                  \right ) .
	\end{align}

\end{proof}

Now we have everything in place to proof Theorem \ref{main_theorem}. The proof resembles the respective one of Theorem 4.1 in \cite{change_of_speed_1} and will be slightly modified and added for the sake of completeness.

\begin{proof}[Proof of Theorem \ref{main_theorem}]
	
	For the identification of the limit equation in the moving frame we will use an analogue of a Lyapunov function argument given in \cite{FifeMcLeod}. Lyapunov functions are a well known and very helpful tool for investigating the long-term behaviour of parabolic partial differential equations.

	Let us define the Lyapunov function as 
	\begin{align}
	\cL[u](t) := \int \limits_\R e^{-c z} \left (  \frac{1}{2} | \partial_{z} u |^2 -F(u) +H(z) F(1)     \right ) \dx{z}  ,
	\end{align}
	where $F(s) := \int_0^s f(\sigma) \dx{\sigma}$ and $H$ is the heaviside-function.
	
	To ensure integrability in the definition of $\cL$ we cut $u$ off as follows
	\begin{align}
	w(t,z) &= u(t,z) &&\text{ for } |z| \leq m t, \\
	w(t,z) &= 0 &&\text{ for } z \leq -m t -1, \\
	w(t,z) &=1 &&\text{ for } z \geq m t +1,
	\end{align}
	for some $m >0$ to be specified later. And we assume $w$ to be smoothed out in a manner, such that Lemma \ref{apriori_bounds_u} still holds for $w$.	
	
	Employing Lemma \ref{apriori_bounds_u} we find that
	\begin{align}
	|\cL  [w]| \leq C  \int \limits_{-m t -1}^{m t +1} e^{-cz} \left ( e^{c z -2 \sigma |z| } + e^{-2\omega t} \right )  \dx{z}
	\end{align}
	which is uniformly bounded for all $t>0$ if  $m>0$ is chosen such that $c m -2\omega <0$. Let us choose $m$ such that $0<m < \frac{1}{2} \min \left \{ \frac{2 \omega}{c},c   \right \}$.
	
	Using integration by parts it follows
	\begin{align}
	\dot{\cL}[w](t) = - \int \limits_\R e^{-cz} \left (  -c \partial_z w + \partial_{zz} w +f(w)       \right ) \partial_t w  ~  \dx{z} .
	\end{align}
	
	Unfortunately, $w$ does not solve $\partial_t w = -c \partial_z w + \partial_{zz} w + f(w)$ and we do not get a sign for $\dot{\cL}$. This is why we try to control the error against
	\begin{align}
	Q[w] = \int \limits_\R e^{-c z}  \left (  \partial_ {zz} w -c \partial_z w +f(w)  \right )^2  \dx{z} ,
	\end{align}
	i.e.
	\begin{align} \label{error_derivative_Lyapunov_function}
	&\dot{\cL}[w](t) + Q[w](t) = \\ &- \int \limits_\R e^{-c z} \left (  - c \partial_z w + \partial_{zz} w +f(w)   \right ) \left (  \partial_t w - \partial_{zz} w + c \partial_z w -f(w)   \right )  \dx{z}
	\end{align}
	
	Note that for $|z| \leq mt$ $w$ solves
	\begin{align}
	\partial_t w - \partial_{zz} w +c \partial_z w -f(w) = -k(z-ct) \partial_z w
	\end{align}
	and that $ k(z-ct)=0$  if $ t\geq \frac{1+x_0}{c-m}$ in $\set {|z| \leq mt}$ .

	For $t \geq \frac{1+x_0}{c-m}$ the last factor in the integral in \eqref{error_derivative_Lyapunov_function} vanishes in $\set {|z| \leq mt}$ and for $  \set {|z| \in (mt,mt+1] }$ we can use the growth estimates from Lemma \ref{apriori_bounds_u}. With all that we can conclude that
	\begin{align}
	\lim \limits_{t \rightarrow \infty} | \dot{\cL}[w](t) +Q[w](t)| =0 .
	\end{align}
	Since $Q[w] \geq 0$ this implies that
	\begin{align}
	\limsup \limits_{t \rightarrow \infty} \dot{\cL}[w](t) \leq 0 .
	\end{align}
	Hence there must be a subsequence $(t_k)_{k \in \N}$, $t_k \rightarrow +\infty$ for $k \rightarrow \infty$ such that
	\begin{align}
	\lim \limits_{k \rightarrow \infty} \dot{\cL}[w](t_k) =0
	\end{align}
	because otherwise $\cL[w]$ could not be uniformly bounded in $t$.
	Therefore it must hold along that subsequence
	\begin{align} \label{limit_of_Q_along_sequence}
	\lim \limits_{k \rightarrow \infty} Q[w](t_k) =0 .
	\end{align}
	
	By Lemma \ref{apriori_bounds_u} and an Arzela-Ascoli argument for a further subsequence (again denoted by $(t_k)_{k \in \N}$) there is a function $u_\infty$ such that:
	\begin{align}
	u( t_k,\cdot) \rightarrow u_\infty(\cdot) \quad \text{ for } k \rightarrow \infty \text{ in } C^2(\R), \\
	w(t_k, \cdot ) \rightarrow u_\infty(\cdot) \quad \text{ for } k \rightarrow \infty \text{ in } C^2(\R) . 
	\end{align}
	Therefore  since $Q\geq 0$ and \eqref{limit_of_Q_along_sequence} for any finite interval $I \subset \R$:
	\begin{align}
 \left (  \int \limits_{I}  e^{-c z} \left ( \partial_{zz}w -c \partial_z w +f(w) \right )^2  \dx{z}   \right )(t_k) &\to 0  \quad \text{ and } \\
	 \left (  \int \limits_{I}  e^{-c z} \left ( \partial_{zz}w -c \partial_z w +f(w) \right )^2  \dx{z}   \right )(t_k) &\to \int \limits_{I} e^{-c z} \left (  \partial_{zz} u_\infty -c \partial_z u_\infty +f(u_\infty)  \right )^2  \dx{z} 
	\end{align}
	as $k \to \infty$.
Hence $u_\infty$ solves
	\begin{align}
	\partial_{zz} u_\infty -c \partial_z u_\infty + f(u_\infty) =0 \quad \text{ a.e. in } \R \text{ and } \\
	\lim \limits_{z \rightarrow -\infty} u_\infty(z) = 0 , \lim  \limits_{z \rightarrow \infty} u_\infty(z) =1
	\end{align}
	
	By uniqueness of traveling fronts up to translation in $z$ (see e.g. \cite{traveling_fronts_in_cylinders} Thm 7.1 or \cite{FifeMcLeod} Cor 2.3), there is $\beta \in \R$ such that
	\begin{align}
	u_\infty(z) = \phi(z+\beta) .
	\end{align}
	
	Now the stability result Lemma \ref{lemma_stability} implies that 
	\begin{align}
	u(t,z) \rightarrow \phi(z+ \beta) \text{ uniformly in } \R \text{ as } t \rightarrow +\infty ,
	\end{align}
	not only for the special subsequence $(t_n)_{n \in \N}$.
		It follows directly that
	\begin{align}
	u(t,x) -\phi(x+ct+\beta) \rightarrow 0 \text{ uniformly in } [-mt-ct,mt-ct]  \text{ as } t \rightarrow \infty.
	\end{align}
	From Lemma \ref{apriori_bounds_u} it follows that 
	\begin{align}
	u(t,x) -\phi(x+ct+\beta) \rightarrow 0 \text{ uniformly in } \R \setminus [-mt-ct,mt-ct]  \text{ as } t \rightarrow \infty.
	\end{align}
	
	This was to be proven.

\end{proof}

\appendix

\section{Remarks on existence and uniqueness}\label{section:existence_and_uniqueness}

 As mentioned before existence and uniqueness  for solutions of \eqref{DiffEqu} can be obtained almost literally copying the proof of Theorem 2.1 in \cite{MatanoObst} or Appendix A in \cite{change_of_speed_1}.

The method of proof therein rests on the construction of suitable super- and subsolutions and repeated application of the maximum principle in order to pass in a sequence of solutions of classical initial value problems, which are monotonely increasing as the starting time decreases, to an entire solution of \eqref{DiffEqu}.

For later reference let us only mention the super- and subsolutions of \eqref{DiffEqu} used in the proof. The supersolution $w^+$ is given by
\begin{align}
w^+(t,x) = 
\begin{cases}
\min \{\phi(x+c t + \xi(t) ) + \phi(-x +c t + \xi(t)),1\} &\text{if } x \geq 0, \\
\min \{  2 \phi(c+\xi(t)),1      \}  &\text{if } x <0
\end{cases}
\end{align}
and the subsolution $w^ -$ is given by
\begin{align}
w^ -(t,x) &:= \sup \limits_{s <0} \tilde{w}^-(t+s,x) \text{ ,where} \\
\tilde{w}^-(t,x) &= 
\begin{cases}
\max \{\phi(x+c t - \xi(t) ) - \phi(-x +c t - \xi(t)),0\} &\text{if } x \geq 0, \\
0  &\text{if } x <0.
\end{cases} 
\end{align}
Here $\xi$ is the solution of the ordinary differential equation
\begin{align}
\dot{\xi}(t) &= M e^{\lambda(c+\xi)} \text{ in } t<\tilde{T} \text{ and } \\
\xi(-\infty) &=0 ,
\end{align}
$\lambda >0$, $\tilde{T}<0$ and $M>0$ only depend on $f$. There is $T<\tilde{T}$ such that
\begin{align}
\tilde{w}^-(t,x) \leq {u}(t,x) \leq w^+(t,x) \quad  \text{ for all } (t,x) \in (-\infty, T]\times \R.
\end{align}


\nomenclature{$\mathcal{L}^N$}{The $N$-dimensional Lebesgue measure}%
\nomenclature{$B_R(x)$}{The open ball of radius $R$ centred in $x$}%
\nomenclature{$C_R(x)$}{The open cube of side length $R$ centred in $x$}%
\nomenclature{$\omega_N$}{The $N$-dimensional Lebesgue-measure of the $N$-dimensional unit ball}%

%
%

\bibliographystyle{abbrv}
\bibliography{Change_of_speed-3.bib}

\end{document}